\documentclass{daj}

\RequirePackage[l2tabu, orthodox]{nag}
\usepackage{microtype}
\usepackage[english]{babel}
\usepackage{amsrefs}
\usepackage{amsthm}
\usepackage{amssymb}
\usepackage{amsmath}
\usepackage{enumerate}
\usepackage{flexisym}

\newcommand{\R}{\mathbb{R}}

\newcommand{\Z}{\mathbb Z}
\newcommand{\RS}{\mathbb{S}}

\newcommand{\supp}{\mathsf{supp}}
\renewcommand{\ker}{\mathsf{ker}}
\DeclareMathOperator{\Dom}{\mathsf{Dom}}
\newcommand{\tworoot}{\mathsf{2root}}
\DeclareMathOperator{\rot}{\mathsf{rot}}

\newcommand{\Diff}{\mathsf{Diff}^1}
\newcommand{\HR}{\mathsf{Homeo}}
\newcommand{\GL}{\mathsf{GL}}

\newcommand{\BS}{\mathsf{BS}}
\renewcommand{\vsp}{\vspace{0.2cm}}

\newtheorem{thm}{Theorem}[section]
\newtheorem{thmA}{Theorem}
\newtheorem{cor}[thmA]{Corollary}

\newtheorem{claim}{Claim}
\newtheorem*{claim*}{Main Claim}

\newtheorem{prop}[thm]{Proposition}
\newtheorem{lem}[thm]{Lemma}

\theoremstyle{definition}

\theoremstyle{remark}
\newtheorem{ex}[thm]{Example}
\newtheorem{rem}[thm]{Remark}

\dajAUTHORdetails{%
  title = {One-dimensional Actions of Higman's Group}, 
  author = {Crist\'{o}bal Rivas and Michele Triestino},
  plaintextauthor = {Crist\'{o}bal Rivas, Michele Triestino},
  keywords = {Higman's group, orderable groups, one-dimensional dynamics, Kaplansky's zero divisors conjecture},
}   

\dajEDITORdetails{%
   year={2019},
   number={20},
   received={16 May 2019},   
   published={23 December 2019},  
   doi={10.19086/da.11151},       
}   

\begin{document}

\begin{frontmatter}[classification=text]

\title{One-dimensional Actions of Higman's Group} 

\author[rivas]{Crist\'{o}bal Rivas\thanks{Partially supported by FONDECYT 1181548.}}
\author[tries]{Michele Triestino\thanks{Partially supported MathAmSud RGSD  (Rigidity and Geometric Structures in Dynamics) 19-MATH-04, the project ANR Gromeov (ANR-19-CE40-0007), and the project ``Jeunes
G\'eom\'etres'' of F. Labourie (financed by the Louis D. Foundation).}}

\begin{abstract}
We build a faithful action of Higman's group on the line by homeomorphisms, answering a question of Yves de Cornulier. As a by-product we obtain many quasimorphisms from Higman's group into the reals. We also show that every action by $C^1$-diffeomorphisms of Higman's group on  the line or the circle is trivial.
\end{abstract}
\end{frontmatter}

\section{Introduction}

The first example of an infinite, finitely generated and simple group was built by Higman \cite{Higman} as a quotient of the group with the following presentation
\begin{equation}\label{eq Higman} \langle \, a_i \;\; (i\in \Z/4\Z ) \mid a_ia_{i+1}a_i^{-1}=a_{i+1}^2 \;  \rangle=H.\end{equation}
Indeed, Higman observed that $H$ is infinite, in fact torsion free, and moreover that if four elements in a group satisfy the relations above and each of them has finite order, then they all have to be the group identity. In particular $H$ has no non-trivial finite quotient, 
and the quotient of $H$ by a maximal normal subgroup is infinite, simple and finitely generated. With time, the name \emph{Higman's group} was attached to the group with the above presentation.

The fact that $H$ has no non-trivial finite quotient implies that every finite dimensional linear representation $\rho: H\to \GL(V)$ is trivial, since finitely generated groups of matrices are residually finite by a theorem of  Mal\textprime cev \cite{malcev}.  It also implies that every  $C^1$-action $\rho:H\to \Diff(M)$ having a periodic orbit is trivial if $M$ is a connected manifold. Indeed, since $H$ has no finite quotient other than the trivial, the finite orbit must be a global fixed point and, since $H$ has only trivial linear representations, the derivative of every $h\in H$ must be trivial at that point. It then follows from the Thurston Stability Theorem \cite{thurston} that if $\rho(H)$ is not the trivial group, then $\rho(H)$ (and hence $H$) must admit a surjective homomorphism onto $\Z$, which is certainly impossible since $H$ has trivial abelianization\footnote{The fact that $H$ is finitely generated and has trivial abelianization implies that $H$ is \emph{not} {(locally) indicable}. Curiously, this notion was introduced by Higman to study zero divisors and units in group ring \cite{Higman-ring}. Compare with our Corollary \ref{c:ZDC}. The fact that locally indicable groups are left-orderable is now well known and was first proved by Burns and Hale in \cite{burns-hale}.}. This immediately implies, for instance, that every representation $\rho:H\to \Diff([0,1])$ is trivial since $0$ is always in a finite orbit.

The main purpose of this work is to show that, despite the constrains described above, Higman's group can be faithfully represented as a group of orientation-preserving homeomorphisms of the real line\footnote{Note that, since $H$ has no finite index subgroup, every action of $H$ by homeomorphisms of an oriented manifold must preserve the orientation.}. This implies, in particular, that $H$ admits faithful representations inside the group of homeomorphisms of any given manifold. Indeed, choose a ball $B^d$ in a manifold $M^d$, the complement of a point $p\in B^d$ in $B^d$ is homeomorphic to $\R\times \mathbb S^{d-1}$; given an action of $H$ on $\R$, extend it to an action on $\R\times \mathbb S^{d-1}$, so that the action on the $\mathbb S^{d-1}$-factor is trivial: this gives an action of $H$ on $B^d$, which fixes $p$ and the boundary $\partial B^d$, so that it can be extended as the identity outside $B^d$.
 
We start by showing, in Section \ref{sec construction}, that $H$ admits a non-trivial action on the real line. This amounts to find 
four homeomorphisms $a$, $b$, $c$, $d$ of the line satisfying 
\[aba^{-1}=b^2 \; , \; bcb^{-1}=c^2\;,\; cdc^{-1}=d^2 \; , \; dad^{-1}=a^2.\]
In our construction, $b$ and $d$ will be \emph{lifts} to the real line of strong parabolic homeomorphisms of the circle, so that the subgroup $\langle b,d\rangle$ display ping-pong dynamics and acts without global fixed points. The choice of the generators $a$ and $c$ is more subtle and it relies on the fact that the Baumslag-Solitar group $\BS(1,2)=\langle \alpha,\beta\mid \alpha \beta\alpha^{-1}=\beta^2\rangle$ admits two different  and easy-to-describe actions on the real line (see the beginning of Section \ref{sec construction}).  We will build $a$ and $c$ as the limit of sequences of homeomorphisms $(a_n)_{n\geq 0}$ and $(c_n)_{n\geq 0}$ respectively, where the relations 
\[a_nba_n^{-1}=b^2\quad \text{and}\quad c_ndc_n^{-1}=d^2,\]
hold for even $n$'s, and the relations 
\[bc_nb^{-1}=c_n^2\quad \text{and}\quad da_nd^{-1}=a_n^2,\]
hold for odd $n$'s. At each step, we will perform slight modifications of the homeomorphisms so that, \emph{in the limit}, both sets of relations will be satisfied. This provides a non-trivial representation 
\[\varphi: H \to \HR_+(\mathbb R).\]

The fact that $\varphi:H\to\HR_+(\R)$ is non-trivial entails that $\varphi(H)$ is a {\em left-orderable} group, that is $\varphi(H)$ admits a total order $\preceq$ such that $\varphi(f)\preceq \varphi(g)$ implies that $\varphi(hf)\preceq \varphi(hg)$ for every $f,g,h\in H$.  In fact, for countable groups being left-orderable is equivalent to admit an injective representation into $\HR_+(\R)$ (see, for instance, \cite{Clay-Rolfsen, ghys}). In our case, we are unable to decide whether $\varphi$ is injective or not, but, using Kurosh Theorem \cite{serre}, in Section \ref{sec faithful} we will show that the kernel of $\varphi$ must be a free group, which is one of the basic examples of left-orderable groups \cite{Clay-Rolfsen, ghys}. In particular, we can {\em lexicographically} extend a left-order on $\varphi(H)$ to a left-order on $H$ (see Section \ref{sec faithful}) so that we can deduce:

\begin{thmA}\label{t:LO}
	Higman's group $H$ is left-orderable and hence it admits a faithful action on the real line by homeomorphisms.
\end{thmA}

As a direct consequence of left-orderability (see \cite[\S 1.8]{Clay-Rolfsen}), we get that the Kaplasky Zero-Divisor Conjecture holds for Higman's group:

\begin{cor}\label{c:ZDC}
If $R$ is a ring without zero-divisors, then the group ring $R[H]$ of Higman's group has no zero-divisor.
\end{cor}

We point out that all the homeomorphisms in our construction of $\varphi:H\to \HR(\R)$  actually commute with the translation $T:x\mapsto x+2$. Thus, we get a faithful action of  $H$ on the circle $\R/_{T(x)\sim x}$ without global fixed points. See Proposition \ref{prop circle action}.
With the same idea of construction, we are able to describe countably many different (i.e.~pairwise non-semi-conjugate\footnote{Recall that two actions $\psi_1:G\to \HR_+(\R)$ and $\psi_2:G\to \HR_+(\R)$ are semi-conjugate, if there is a proper, non-decreasing map $\alpha:\R\to \R$ such that $\alpha\circ \psi_1(g)=\psi_2(g)\circ \alpha$ for all $g\in G$. Two actions on the circle are semi-conjugate if they have lifts to the real line which are semi-conjugate. See \cite{ghys,mann,bfh} for further details.}) actions of $H$ on the circle, and uncountably many different actions on the real line, all without global fixed points. See Section \ref{sec quasi morphisms}.

\begin{rem}
	We have already mentioned that Higman's group $H$ is perfect (that is, its abelianization is trivial), so every homomorphism $\Phi:H\to \R$ is trivial. A weaker notion is that of \emph{quasimorphism} to $\R$, that is, a map $\Phi:H\to \R$ for which there exists $D>0$ such that for every $g,h\in H$, one has
	\[|\Phi(gh)-\Phi(g)-\Phi(h)|\le D.\]
The set of all quasimorphisms $H\to \R$ defines the real vector space $\mathcal{QM}(H)$. By cohomological considerations, one deduces that the dimension of $\mathcal{QM}(H)$ is the cardinal of continuum: indeed, for perfect groups, the quotient $\mathcal{QM}(H)/\{\text{bounded functions}\}$ is isomorphic to the second bounded cohomology group $H^2_b(H;\R)$, whose dimension is the cardinal of continuum since $H$ is a non-trivial amalgam \cite{Fujiwara}. Similarly, one can consider the second bounded cohomology group $H^2_b(H;\Z)$ with integer coefficients. Every action on the circle without global fixed points determines a non-trivial cocycle in $H^2_b(H;\Z)$, by the so-called \emph{bounded Euler class}. Moreover, two actions  of $H$ are semi-conjugate if and only if their bounded Euler classes in $H^2_b(H;\Z)$ are the same \cite{ghys}. Therefore, the actions constructed in Section \ref{sec quasi morphisms} provide countably many linearly independent classes in  $H^2_b(H;\Z)$. See the references \cite{ghys,frigerio} for more about bounded cohomology related to group actions on the circle.
\end{rem}


In contrast, although our actions have no fixed points neither periodic orbits, we can still show that they cannot be made differentiable. In fact we have:

\begin{thmA}\label{t:noC1} 
	Every representation $\rho:H\to \Diff(\R)$ is trivial. The same holds for every representation $\rho:H\to \Diff(\mathbb S^1)$.
\end{thmA}


Our proof relies on the knowledge of $C^1$-actions of the Baumslag-Solitar group
\[\BS(1,2)=\langle \alpha, \beta\mid \alpha \beta \alpha^{-1}=\beta^2\rangle\]
on the line \cite{BMNR}  together with an analysis of the possible combinatorics of fixed points of elements of $H$ when acting on the line by homeomorphisms. The proof is given in Section \ref{sec smooth} and it is independent of the rest of the paper.

\section{The construction of $\varphi$}
\label{sec construction}

\subsection{Preliminaries}

Higman's group contains four canonical copies of the Baumslag-Solitar group 
\[\BS(1,2)=\langle \alpha, \beta\mid \alpha \beta\alpha^{-1}=\beta^2\rangle.\]
This latter group acts on the real line via its \emph{standard affine action}  $\alpha:x\mapsto 2x$, $\beta:x\mapsto x+1$ (see Figure~\ref{fig:Standard_affine}, left), which is a faithful action since every proper quotient of $\BS(1,2)$ is abelian. \emph{Non-standard affine} actions of $\BS(1,2)$ on the line can be obtained from the standard one by blowing up an orbit, following the method of Denjoy (see  Figure~\ref{fig:Standard_affine}, right). In fact, it is not very hard to see that \emph{if there is} a non-trivial action of $H$ on the line, then a non-standard affine action must appear at least locally (cf.\ the proof of Theorem \ref{t:noC1} in Section \ref{sec smooth}). The following example is the one that we will use.

\begin{ex}[Denjoy trick]\label{ex denjoy} Consider the affine maps of the real line $f(x)=2x$ and $g(x)=x+1$, which generate a group isomorphic to $\BS(1,2)$. The orbit of $0$ coincides with the dyadic rationals $\Z[1/2]=\{n/2^m\mid n\in \Z, m\geq 0\}$. We blow up the orbit of $0$ to build a \emph{new} real line, that is, we replace every point $o\in \Z[1/2]$ of the orbit of $0$ by a closed interval $J_o$ in such a way that, for every compact subset $K$,  the sum $\sum_{o\in K} |J_o|$ is bounded (we are denoting by $|J|$ the length of the interval $J$). For our purpose, if $n/2^m$ is a reduced fraction, we  set $|J_{n/2^m}|=\kappa/2^{2^m}$, for some constant $\kappa>0$ to be fixed later. By convention the reduced fraction  of $0$ is $0/1$, so $|J_0|=\kappa/4$. 

The maps $f$ and $g$ act naturally on the boundary points of the collection of intervals $\{J_o\}$, and this action can be extended to the \emph{new} real line by taking linear interpolation and passing to the closure of the union of intervals. We put the new origin (still denoted by $0$) at the midpoint of $J_0$. The resulting maps are called $\bar f$ and $\bar g$.  Observe that since $n/2^m$ is a reduced fraction if and only if $1+n/2^m=(2^m+n)/2^m$ is a reduced fraction, we have that the homeomorphism $\bar g$ preserves the length of the intervals it permutes and hence $\bar g$ is still a translation. By adjusting $\kappa$, and possibly replacing $\bar g$ by some power of it, we can assume that $\bar g$ is the translation by  2 and that $|J_0|\leq 1/4$. 

Finally, we let $M\subseteq \R$ be the minimal invariant set for $\langle \bar f ,\bar g\rangle$ action on the line. That is, $M$ is closed, $\langle \bar f,\bar g\rangle$-invariant, and contains no proper, closed, $\langle \bar f,\bar g \rangle$-invariant subsets. In the present case, since $J_0$ is a wandering interval, and no $\langle \bar f ,\bar g\rangle$-orbit is discrete, the set $M$ is locally a Cantor set which intersects trivially the $\langle \bar f ,\bar g\rangle$-orbit of $J_0$ (see \cite{ghys}).
\end{ex}

\begin{rem} Observe that $\bar f$ from Example \ref{ex denjoy} is the identity on $J_0$. For our purpose it will be important to replace $\bar f$ by another map which is not the identity on $J_0$, while preserving the action on the minimal set $M$ from the original action (Figure~\ref{fig:Standard_affine}, right). This will be done in Lemma \ref{lem 1} below. 
\end{rem}

\begin{figure}[t]
	\[
	\includegraphics[scale=1]{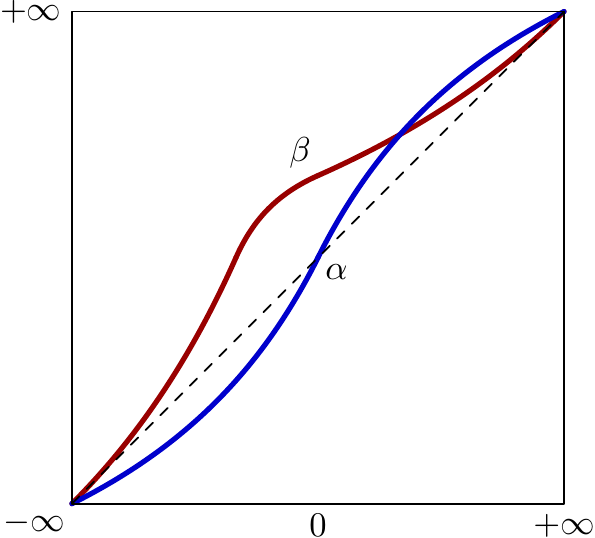}
	\hspace{2cm}
	\includegraphics[scale=1]{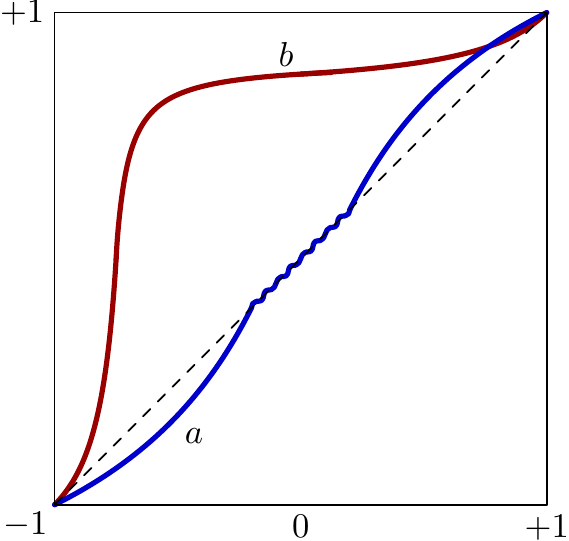}
	\]
	\caption{The standard affine action of $\BS(1,2)$ (left) and an action \emph{\`a la Denjoy} (right).}\label{fig:Standard_affine}
\end{figure}

Crucially,  $\BS(1,2)$ admits a different family of actions without fixed points on the real line. These are  the so-called \emph{Conradian actions}, which are actions where $\alpha$ acts without fixed points and $\beta$ fixes a fundamental domain of $\alpha$ \footnote{For the purpose of this note, this can be taken as a definition of Conradian action of $\BS(1,2)$. In general, a Conradian action of a  group is a faithful action where every induced order is Conradian. Equivalently, it is an action without crossed elements, see \cite{navas} for details.  In \cite{rivas} it is shown that  any faithful action of $\BS(1,2)$  without global fixed points is either of affine type or Conradian.}  (see Figure~\ref{fig:Conrad}). The following lemma formalizes the idea that for building a Conradian action one just need pick $\alpha\in \HR_+(\R)$ and \emph{any other} homeomorphism defined on a fundamental domain of $\alpha$ which is then {\em extended} to a homeomorphism of the line using the relation of $\BS(1,2)$. For the statement, we denote by $\HR_0(I)$ the group of homeomorphisms of an open subset $I\subset \R$ which are isotopic to the identity. Also, given a homeomorphism $h:I\to I$, we define its \emph{support} as the open subset $\supp(h)=\{x\in I\mid h(x)\neq x\}$.

\begin{figure}[t]
	\[
	\includegraphics[scale=1]{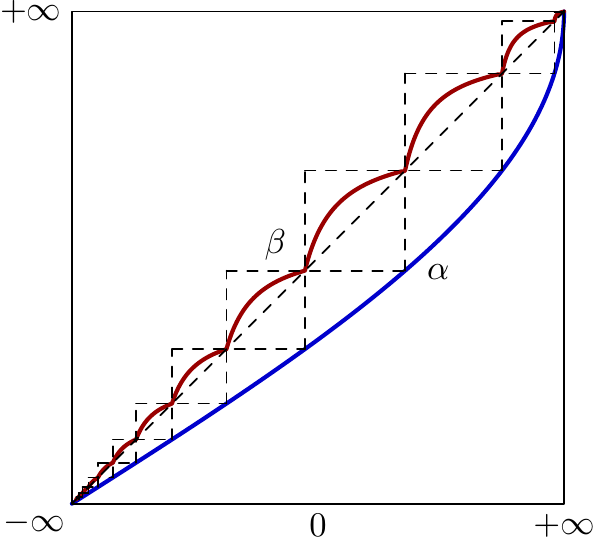}
	\]
	\caption{An example of Conradian action of $\BS(1,2)$ on the real line.}\label{fig:Conrad}
\end{figure}

\begin{lem} \label{lem 2}Given $f\in \HR_+(\R)$, an open set $I$ such that $f( I )\cap I=\emptyset$, and $h\in \HR_0(I)$, there exists a map $R_1:(f,I,h)\mapsto  g_h\in \HR_+(\R)$ such that \begin{enumerate}[(i)]
\item $g_h(x)=h(x)$ for every $x\in I$, \label{2i:1}
\item $g_h$  is the identity outside $ \bigcup_{n\in \Z} f^n(I)$,
\item $fg_hf^{-1}(x)=g_h^2(x)$ for every $x\in \R$.\label{ii:3}
\end{enumerate}
\end{lem}

\begin{proof}
We need to define $g_h$. Observe that the condition $f(I)\cap I=\emptyset$ implies that the images $f^n(I)$ are pairwise disjoint. On $f^{-n}(I)$ ($n>0$) we define $g_h$ as $f^{-n}h^{2^n}f^n$ and on $f^n(I)$ we want to define $g_h$ as $f^n\, \tworoot^{(n)}(h)f^{-n}$, where $\tworoot(h)$ is the square root of $h$ and $\tworoot^{(n)}=\tworoot\circ \cdots \circ \tworoot$ $n$-times.

Certainly, every $h\in \HR_+(\R)$ has a square root in $\HR_+(\R)$. For instance,  if $h$ has no fixed points, then it is conjugate to a translation and $\tworoot(h)$ can be then defined as half of the translation conjugated back, while for general $h$ one restricts the attention to connected components of the support of $h$ and repeats the argument locally. The problem is that the square root is highly non-unique (for instance, the translation to which a homeomorphism $h$ without fixed points is conjugate to is not unique), so $\tworoot$ is not well defined (so far). 

To give a precise definition of $\tworoot:\HR_0(I)\to \HR_0(I)$  (and hence a precise definition of $g_h$) we proceed as follows. For every connected component $C$  of the support of $h\in \HR_0(I)$, let $x_0\in C$ be its midpoint (this is well-defined, as from $f(I)\cap I=\emptyset$, every component $C$ is a bounded interval). Assume $h(x_0)>x_0$ (the case $h(x_0)<x_0$ is analogous). Then define a measure $\mu_0$ on $[x_0,h(x_0))$,  as the push forward of the Lebesgue measure $\mathsf{Leb}$ on $[0,1)$ by the unique affine map $A$ sending $[0,1)$ to $[x_0,h(x_0))$ bijectively. So $\mu_0(X)=\mathsf{Leb}(A^{-1}(X))$ for every Borel subset $X\subset [x_0,h(x_0))$. Then there is a unique way of extending $\mu_0$ to an $h$-invariant measure over $C$. The resulting measure $\mu_C$ is fully supported on $C$, gives finite mass to compact subsets of $C$, and has not atoms. We can now use $\mu_C$ to conjugate the restriction of $h$ to $C$ to the translation by 1 on the real line. Indeed if we let $F(x)=\mathsf{sign}(x-x_0)\mu([x_0,x))$, then $F h F^{-1}$ is the translation by 1. Then we can take the square root inside the group of translations and conjugate back. This is our preferred $\tworoot$.
\end{proof}

\subsection{Choosing $b$ and $d$} \label{sec b and d}
 Let $J_0$ and $\bar g:x\mapsto x+2$ be as in Example \ref{ex denjoy}, and let $\psi:\R\to (-1,1)$ be the homeomorphism given by the linear interpolation of 
\[\psi(n)=\left\{\begin{array}{cr} 0 & \text{if }n=0,\\ \sum_{i=1}^n 2^{-i} & \text{if }n>0, \\ -\sum_{i=1}^{-n} 2^{-i} & \text{if }n<0, \end{array}\right.\]
for $n\in \Z$. Let $\beta_0=\psi \circ \bar g \circ \psi^{-1}\in \HR_+(-1,1)$, and call $I_0$ the image of $J_0$ under $\psi$. Call $I_1=I_0+1$ and define $\beta_1\in \HR_+(0,2)$ by $\beta_1(x)=\beta_0(x-1)+1$. We note that $I_0$ and $I_1$ are disjoint intervals, one centered at $0$, and the other centered at $1$; we note also that the maps $\beta_0$ and $\beta_1$ are piecewise linear.
Finally, we extend $\beta_0$ and $\beta_1$ to homeomorphisms of the real line by imposing that they commute with the translation $T:x\mapsto x+2$. The resulting homeomorphisms are still denoted $\beta_0$ and $\beta_1$. We also let $\mathcal{I}_0=\bigcup_{n\in \Z} I_0+2n$ and $\mathcal{I}_1=\bigcup_{n\in \Z} I_1+2n$, which are disjoint open sets, and choose $N>0$ such that 
\begin{equation}\label{eq ping}\beta_0^{\pm N}(\R\setminus \mathcal{I}_1)\subseteq \mathcal{I}_1 \quad\text{and}\quad \beta_1^{\pm N}(\R\setminus \mathcal{I}_0)\subseteq \mathcal{I}_0.\end{equation}
By possibly enlarging $N$, we can (and will) also assume that for all $n\geq N$ we have 
\begin{equation}\label{eq clave}(\beta_i^{\pm n})'(x)\leq 1/2 \quad\text{for all $x\notin \mathcal{I}_{i+1}$,}\end{equation} 
where the indices are taken in $\Z/2\Z$. Note that by (\ref{eq ping}) and the ping-pong lemma, the group $\langle \beta_0^N,\beta_1^N\rangle$ is a free group inside $\HR_+(\R)$. We   let $b=\beta_0^N$ and $d=\beta_1^N$.

\subsection{Choosing $a$ and $c$} \label{sec a and c}

For $f$ and $g$ homeomorphisms of the line, we let
\[|f-g|_\infty=\sup_{x\in \R}|f(x)-g(x)|+\sup_{x\in \R}|f^{-1}(x)-g^{-1}(x)|.\]
This  ``norm'' may be infinite but in many cases it is bounded, so it defines a true distance. This happens for instance in restriction to the subgroup 
\[{\HR_{+,T}(\R)}=\{f\in \HR_+(\R)\mid T\circ f=f\circ T\},\]
where $T$ is a translation (which is our case). Endowed with $|\,\cdot\, |_\infty$, ${\HR_{+,T}(\R)}$ becomes a complete metric space, which makes it a Polish group (see \cite[Proposition I.2.4]{Herman}).
We will build $a$ and $c$ as the limit (in this topology) of converging sequences of homeomorphisms $(a_n)_{n\geq 0}$ and $(c_n)_{n\geq 0}$ respectively; for every $n\ge 0$, the homeomorphism $a_n$ (resp.~$c_n$) will be defined as a perturbation of $a_{n-1}$ (resp.~$c_{n-1}$). 
 We keep the notation from \S \ref{sec b and d}.

Let $\bar f$ and $\bar g$ be the homeomorphisms from Example \ref{ex denjoy} and $M$ its minimal set. Let $a_{-1}=\psi \circ \bar f \circ \psi^{-1}\in \HR_+(-1,1)$, and $c_{-1}\in \HR_+(0,2)$ be defined by $c_{-1}(x)=a_{-1}(x-1)+1$. We extend $a_{-1}$ and $c_{-1}$ inside ${\HR_{+,T}(\R)}$  by imposing that they commute with the translation $T:x\mapsto x+2$. We denote the resulting elements by $a_0$ and $c_0$ respectively. Observe that  $a_0ba_0^{-1}(x)=b^2(x)$ for all $x\in \R$  and in the same way  $c_0 d c_0^{-1}(x)=d^2(x)$ for all $x\in \R$.  Note also that $a_0$ is supported on $\R\setminus \mathcal{I}_0$ and $c_0$ is supported on $\R\setminus \mathcal{I}_1$ (because $\bar f$ is the identity on $J_0$).

\vsp
We now explain how to obtain $a_n$ and $c_n$ for $n\geq 1$. The construction of the sequences $(a_n)$ and $(c_n)$, as well as the proof of its convergence are independent and completely analogous upon exchanging the roles of $b$ and $d$ (and  replacing $0$ indices by $1$ indices). We only explain the case of $a_n$. Let $\mathcal M_0=\bigcup_{n\in \Z} \psi(M)+2n$. We will perform successive perturbations on $a_0$ on the complement of $\mathcal M_0$.

\begin{lem} \label{lem 1}
Let $T:\R\to\R$ denote the translation by $2$. There exists a map $R_2:h\mapsto a_h\in {\HR_{+,T}(\R)}$ defined on elements of $\HR_{+,T}(\R)$ supported on $\mathcal I_0$ with the following properties.
\begin{enumerate}[(i)]
\item $a_h(x)=h(x)$ for every $x\in\mathcal I _0$. \label{i:1}
\item  $a_h(x)=a_0(x)$ for $x\in \mathcal M_0$ and $ a_h b a_h^{-1}(x)=  b^2(x)$ for every $x\in \R$. \label{i:2}
\item $R_2$ is an isometry with respect to the $|\,\cdot\, |_\infty$-norm: for any two $h,h'\in \HR_{+,T}(\R)$ supported on $\mathcal I_0$, one has $| a_h-a_{h'}|_\infty= |h-h'|_\infty$. \label{i:3}
\end{enumerate}
\end{lem}

\begin{proof}	
Suppose that $h\in {\HR_{+,T}(\R)}$ is supported on $\mathcal I_0$. Since $\mathcal I_0$ is contained in a fundamental domain of $b$ (i.e.~$b^n(\mathcal I_0)\cap \mathcal I_0=\emptyset$ for $n\neq 0$),
the infinite composition $h_1:=\prod_{n\in \Z}b^nhb^{-n}$ is a composition of homeomorphisms with pairwise disjoint support, all commuting with $T$. 
Hence $h_1$ is a well-defined element in $\HR_{+,T}(\R)$, supported on $\bigcup_{n\in\Z} b^n (\mathcal I_0)$, which furthermore commutes with $b$. We can then define $R_2(h)=a_h=a_0h_1$ which satisfies item \eqref{i:1}.

For \eqref{i:2} we observe that $h_1(x)=x$ for all $x\in \mathcal M_0$ since $b^n(\mathcal I_0)$ and $\mathcal M_0$ are disjoint (see the end of Example \ref{ex denjoy}) and $h_1$ is supported on $\bigcup b^n(\mathcal I_0)$. In particular $a_h(x)=a_0(x)$ for all $x\in \mathcal M_0$. We can also check that $a_hb(x)=a_0h_1b(x)=a_0bh_1(x)=b^2a_0h_1(x)=b^2a_h(x)$ for all $x\in \R$. 

Finally, to show \eqref{i:3} we note that
\begin{equation}\label{eq:norm1}
	\sup_{x\in\R}|a_0h_1(x)-a_0h_1'(x)| =\sup_{n\in\Z}\sup_{x\in \mathcal I_0}|a_0b^nh(x)-a_0b^nh'(x)|=\sup_{n\in\Z}\sup_{x\in \mathcal I_0}|b^{2n}a_0h(x)-b^{2n}a_0h'(x)|.
\end{equation}
Further, for $x\in\mathcal I_0$, the images  $h(x)$ and $h'(x)$ belong to $\mathcal I_0$, on which $a_0$ is the identity and, as $\mathcal I_0$ and $\mathcal I_1$ are disjoint, condition \eqref{eq clave} implies that the supremum in the last term is achieved for $n=0$.
After these considerations, equality \eqref{eq:norm1} reduces to
\begin{equation*}
\sup_{x\in\R}|a_0h_1(x)-a_0h_1'(x)| =\sup_{n\in\Z}\sup_{x\in \mathcal I_0}|b^{2n}h(x)-b^{2n}h'(x)|=\sup_{x\in \mathcal I_0}|h(x)-h'(x)|.
\end{equation*}
Similarly we control the difference of inverses: we have
\begin{equation*}
\sup_{x\in\R}|(a_0h_1)^{-1}(x)-(a_0h_1')^{-1}(x)|=  \sup_{x\in\R}|(h_1)^{-1}(x)-(h_1') ^{-1}(x)| =\sup_{n\in\Z}\sup_{x\in \mathcal I_0}|b^n(h)^{-1}(x)-b^n(h')^{-1}(x)|
\end{equation*}
which, by \eqref{eq clave},  equals $\sup_{x\in\mathcal I_0}|(h)^{-1}(x)-(h')^{-1}(x)|$.
Putting the two pieces together, we obtain the equality $| a_h-a_{h'}|_\infty= |h-h'|_\infty$.
\end{proof}

We will use sequentially Lemma \ref{lem 2} and Lemma \ref{lem 1} to perturb $a_0$ outside $\mathcal M_0$. To this extent we let
\[\Dom_0=\{f\in {\HR_{+,T}(\R)})\mid  f(x)=a_0(x) \text{ for every $x\in \mathcal M_0$}\}.\]
Observe that $\Dom_0$ is a closed subset of ${\HR_{+,T}(\R)}$  and for $f\in \Dom_0$, we have that $ f(\mathcal I _0)=\mathcal I_0$, and $fbf^{-1}(x)=b^2(x)$ for every $x\in \mathcal M_0$.  In particular, for $f\in \Dom_0$, we can define $\widetilde{R_2}(f)=R_2(f\restriction_{\mathcal I_0})$, where $f\restriction_{\mathcal I_0}$ denotes the restriction of $f$ to $\mathcal I_0$, extended as to be the identity outside $\mathcal I_0$. Certainly $\widetilde{R_2}$ is an idempotent of $\Dom_0$ (i.e.~$\widetilde{R_2}^2=\widetilde{R_2}$) and, from Lemma \ref{lem 1}.\ref{i:3},  it also satisfies
\begin{equation}\label{eq 1lip}|\widetilde{R_2}(f)-\widetilde{R_2}(g)|_\infty= |f\restriction_{\mathcal I_0} - g\restriction_{\mathcal I_0}|_\infty.\end{equation} 
Note also that by (\ref{eq ping}), $d^n(\R\setminus \mathcal I_0)\subset \mathcal I_0$ for all $n\neq 0$ (and hence $d^n(\mathcal M_0) \subset \mathcal I_0$ for all $n\neq 0$). In particular, for $f\in \Dom_0$ we can apply $R_1$ from Lemma \ref{lem 2} to define $\widetilde{R_1}(f)=R_1(d,  \R\setminus\mathcal I_0 ,f\restriction_{\R\setminus \mathcal I_0} )$. Certainly $\widetilde{R_1}$ is also an idempotent of $\Dom_0$ which, by (\ref{eq clave}),  satisfies that
\begin{equation}\label{eq 2lip}|\widetilde{R_1}(f)\restriction_{\mathcal I_0}-\widetilde{R_1}(g)\restriction_{\mathcal I_0}|_\infty\leq \tfrac{1}{2}\;  |f\restriction_{\R\setminus \mathcal I_0} - g \restriction_{\R\setminus\mathcal I_0}|_\infty.\end{equation}
Suppose that $a_n$ is defined for $n\leq 2k$ ($k\geq 0$). We define $a_{2k+1}=\widetilde{R_1}(a_{2k})$ and $a_{2k+2}=\widetilde{R_2}(a_{2k+1})$.

\begin{claim*}
	The sequence $a_n$ converges in $\HR_{+,T}(\R)$.
\end{claim*}

Before giving the proof note that if $a_n$ converges to $a$, then $a$ is both $\widetilde{R_2}$ and $\widetilde{R_1}$ invariant so, by Lemma~\ref{lem 2}.\ref{ii:3} and Lemma \ref{lem 1}.\ref{i:2}, it satisfies $aba^{-1}=b^2$ and $dad^{-1}=a^2$ over $\R$ as desired.

\begin{proof}[Proof of Claim] First observe that, from (\ref{eq 1lip}) and (\ref{eq 2lip}), it follows that $\widetilde{R_2} \widetilde{R_1}$ and $\widetilde{R_1} \widetilde{R_2}$ are contractions (by $1/2$) of the complete metric space $(\HR_{+,T}(\R),|\,\cdot\,|_\infty)$, so by the classical Banach Fixed Point Theorem, both of them have unique (attracting) fixed points in $\Dom_0\subset\HR_{+,T}(\R)$; call them $p=\widetilde{R_2}\widetilde{R_1}(p)$ and $q=\widetilde{R_1}\widetilde{R_2}(q)$. Observe that since $\widetilde R_1$ and $\widetilde R_2$ are idempotents, if follows by uniqueness that $\widetilde R_1(q)=q$ and $\widetilde R_2(p)=p$.  We need to show that $p=q$.
 
By the definition of $\Dom_0$, we have $p(x)=q(x)=a_0(x)$ for all  $x\in \mathcal M_0$ and hence \[|p\restriction_{\R\setminus \mathcal I_0} - q\restriction_{\R\setminus \mathcal I_0}|_\infty\leq |J|,\]
where $J$ is the largest connected component of the complement of $\mathcal M_0\cup \mathcal I_0$.
In particular, applying $\widetilde{R_1}$, we get from (\ref{eq 2lip}) that
 \[|\widetilde{R_1}(p)\restriction_{\mathcal I_0} - q\restriction_{\mathcal I_0}|_\infty \leq \tfrac12|J|,\]
and then, applying $\widetilde{R_2}$, (\ref{eq 1lip}) gives 
 \[|p - \widetilde{R_2}(q)|_\infty\leq \tfrac12|J|.\]
If we now repeat $k$ times this process we get 
\[|\widetilde{R_1}(p)\restriction_{\mathcal I_0} - q\restriction_{\mathcal I_0}|_\infty \leq |J|/2^k\quad \text{and}\quad |p - \widetilde{R_2}(q)|_\infty,\leq |J|/2^k.\]
It follows that $p=\widetilde{R_2}(q)$ and hence $\widetilde{R_1}(p)=q$.  But, respectively, this implies that $p=q$   over $ \mathcal I_0$ (see Lemma \ref{lem 1}.\ref{i:1}) and over $\R\setminus \mathcal I_0$ (see Lemma \ref{lem 2}.\ref{2i:1}). So $p=q$, and we declare $a=p=q$ as needed.
\end{proof}

\begin{rem}\label{rem projection} Observe that the four homeomorphisms $a,b,c,d$ all commute with the translation $T$, so they define homeomorphisms of the circle $\mathbb S^1\cong \R/_{T(x)\sim x}$, and hence a representation  of $H$ into $\HR(\mathbb S^1)$. Further, since $c$ and $d$ do not share fixed points, the resulting action on the circle is without global fixed points. See Figure \ref{fig:Higman}.

\end{rem}
\begin{figure}[t]
	\[
	\includegraphics[scale = 1]{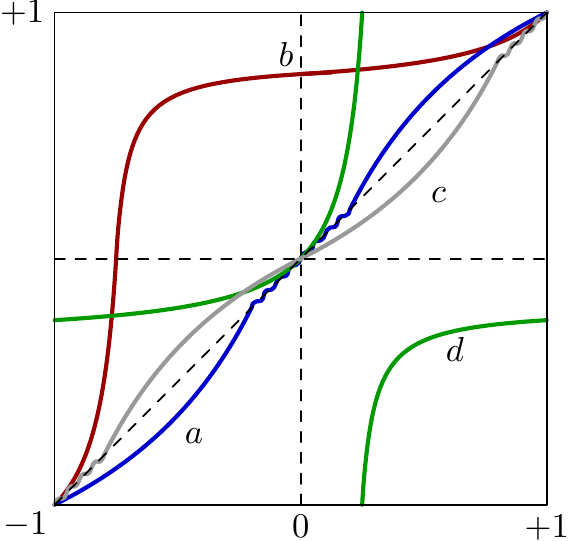}
	\]
	\caption{The action of Higman's group on the circle $\mathbb S^1\cong\R/\langle T\rangle$.}\label{fig:Higman}
\end{figure}

\section{A faithful action}
\label{sec faithful}

In the action $\varphi:H\to {\HR_{+,T}(\R)}$ built in Section \ref{sec construction}, the subgroup $F=\langle b,d\rangle$ acts as a lift of a ping pong action on the circle (with parabolic elements) so $\varphi$ is an embedding in restriction to $\langle b,d\rangle$. This remark, together with Kurosh Theorem \cite{serre}, give the following:

\begin{lem}\label{l:kernel}
	The kernel of the action $\varphi$ is a free group and hence is left-orderable.
\end{lem}

\begin{proof}
The group $H$ is the amalgamated product 
\begin{equation}\label{eq amalgama} H= \langle d,a,b \mid dad^{-1}=a^2\; ,\, aba^{-1}=b^2 \rangle*_{\langle b , d\rangle}\langle b,c,d\mid  bcb^{-1}=c^2\; ,\, cdc^{-1}=d^2 \rangle.\end{equation}
By the preceding remark, $\ker(\varphi)$ has trivial intersection with the edge group $F=\langle b,d\rangle$.
Therefore we can apply Kurosh Theorem to $\ker(\varphi)$ in the amalgamated product \eqref{eq amalgama} (in the case of a normal subgroup, it is enough to look at the intersection with the edge group): it gives that $\ker(\varphi)$ is the free product of a free group with intersections of $\ker(\varphi)$ with conjugates of the vertex groups 
\[H_1=\langle d,a,b \mid dad^{-1}=a^2\; ,\, aba^{-1}=b^2 \rangle\quad\text{and}\quad H_2=\langle b,c,d\mid  bcb^{-1}=c^2\; ,\, cdc^{-1}=d^2 \rangle.\]
It is then enough to prove that the intersection of $\ker(\varphi)$ with every conjugate of $H_1$ and $H_2$ is also a free group. Observe that by normality of $\ker(\varphi)$, it is enough to prove that the intersections $\ker(\varphi)\cap H_1$ and $\ker(\varphi)\cap H_2$ are free.
For this, we repeat nearly the same argument.

Indeed, the group $H_1= \langle d,a,b \mid dad^{-1}=a^2\; ,\, aba^{-1}=b^2 \rangle$ is also an amalgamated product
\begin{equation}\label{eq amalgama2} H_1= \langle d,a\mid dad^{-1}=a^2\rangle *_{\langle a\rangle} \langle a,b \mid  aba^{-1}=b^2 \rangle,\end{equation}
and the two vertex groups $\langle d,a\mid dad^{-1}=a^2\rangle$ and $\langle a,b \mid  aba^{-1}=b^2 \rangle$ are two copies of $\BS(1,2)$ which act faithfully by construction of the action (every proper quotient of $\BS(1,2)$ is abelian, and certainly this is not the case for this action). Therefore the intersection of $\ker(\varphi)$ with the vertex group in $H_1$ is trivial and this implies that $\ker(\varphi)\cap H_1$ is a free group. The same argument works for the intersection with $H_2$. This shows that $\ker(\varphi)$ is a free group (potentially trivial).

Finally, the fact that free groups are left-orderable is well known and can be found in \cite{Clay-Rolfsen} or \cite{ghys}.
\end{proof}

\begin{proof}[Proof of Theorem~\ref{t:LO}]
As a consequence, the group $H$ fits into the short exact sequence
\[
1\to \ker(\varphi)\to H\to \varphi(H)\to 1,
\]
where $\ker(\varphi)$ and $\varphi(H)$ are both left-orderable after Lemma~\ref{l:kernel}. This implies that $H$ is left-orderable as well:
denote by $\preceq_\varphi$ and $\preceq_\ker$ the left-invariant orders on $\varphi(H)$ and $\ker(\varphi)$ respectively; then define the order $\preceq$ on $H$ by declaring
\[id\preceq h\quad\text{if and only if}\quad\begin{cases}
h\in\ker(\varphi)\text{ and }id\preceq_\ker h,\text{ or}\\
h\notin\ker(\varphi)\text{ and }\varphi(id)\preceq_\varphi \varphi(h),
\end{cases}\]
and extending it so that it is left-invariant (declaring $h\preceq g$ if and only if $id\preceq h^{-1}g$; the fact that $\preceq_\varphi$ and $\preceq_\ker$ are left-invariant implies that the definition is consistent, see \cite{Clay-Rolfsen} for more details). 
Finally, every countable left-orderable group is isomorphic to a subgroup of $\HR_+(\R)$  (see \cite{Clay-Rolfsen} or \cite{ghys}), therefore $H$ admits faithful actions on the real line.
\end{proof}

We can also show the analogous result for the circle.

\begin{prop} \label{prop circle action} $H$ admits a faithful action on the circle without global fixed points.

\end{prop}
\begin{proof}[Scketch of proof]
The proof follows the same lines as for Theorem \ref{t:LO}. \begin{enumerate}
\item For countable groups, acting faithfully on the circle by orientation preserving homeomorphisms is equivalent to admitting a circular order invariant under left-multiplications \cite{calegari, mann-rivas}.

\item If $1\to L\to G\to C\to 1$, where $C$ is a circularly left-orderable group and $L$ is a left-orderable group, then $G$ is circularly left-orderable  \cite{calegari}. 
\end{enumerate}
In our case we have that $1\to \ker(\varphi)\to H\to \overline{\varphi}(H)\to 1$, where $\overline \varphi:H\to \HR(\RS^1)$ is the projection of $\varphi$ to $\RS^1\cong \R_{T(x)\sim x}$, see Remark \ref{rem projection}.  Since $\overline \varphi (H)$ acts without global fixed points, the same is true for actions coming from circular orders obtained from the preceding short exact sequence \cite{mann-rivas}. \end{proof}

\section{Many inequivalent actions}
\label{sec quasi morphisms}

For the construction of $\varphi$ in Section \ref{sec construction}, the initial point was the choice of $b$ and $d$ starting from $T$-periodic extensions of  $\beta_0\in\HR_+([-1,1])$ and $\beta_1\in \HR_+([0,2])$. It is possible to get an analogous construction of $\varphi$ by imposing milder periodicity conditions to the extensions of $\beta_0$ and $\beta_1$ to the real line. Concretely, for arbitrary  sequences $(\epsilon_n)$ and $(\delta_n)\in \{-1,+1\}^\Z$, we can consider extensions of $\beta_0$ and $\beta_1$ verifying 
\[\begin{array}{lr}
\beta_0(T^n(x))=T^n(\beta_0^{\epsilon_n}(x)) & \text{for every }n\in\Z\setminus\{0\}\text{ and every }x\in [-1,1],\\
\beta_1(T^n(x))=T^n(\beta_1^{\delta_n}(x)) & \text{for every }n\in\Z\setminus \{0\}\text{ and every }x\in [0,2].\\
\end{array}\]
The effect in the corresponding choice of the new $b=\beta_0^N$ and $d=\beta_1^N$ is that, locally, they could be switched to the inverse of the old $b$ and $d$ (much in the spirit of \cite{hyde-lodha}). This new $b$ and $d$ (and hence the new $a$ and $c$) commute with a translation (and hence project to homeomorphisms of the circle) if and only if $(\epsilon_n)$ and $(\delta_n)$ have a common period. Moreover, any two distinct choices of sequences (which are not shift-equivalent) give actions which are not semi-conjugate one to another.

\section{Smooth actions}
\label{sec smooth}

In this section we give the  proof of Theorem~\ref{t:noC1}. 
We will first establish the result for actions on the real line.

\begin{prop}\label{p:noC1}
Every representation $\rho:H\to \Diff(\R)$ is trivial. 
\end{prop}

Our proof relies on the following:

\begin{lem}\label{l:noC1} Let $\psi:H\to \HR(\R)$ be a non-trivial  representation without global fixed points. Then there is a copy $B$ of the Baumslag-Solitar group $\BS(1,2)$ inside $H$ such that the restriction of $\psi$ to $B$ is faithful, $\psi(B)$ preserves a compact interval $I$, and the restriction of the action of $\psi(B)$ to $I$ is not semi-conjugate to the standard affine action.
\end{lem}

Before proving the lemma, let us explain how to deduce the proposition.

\begin{proof}[Proof of Proposition~\ref{p:noC1}]
As pointed out before, $H$ contains several copies of 
$\BS(1,2)$.
After \cite{BMNR} every action of $\BS(1,2)$ inside $\Diff_+(I)$, where $I$ is a compact interval, must be (topologically) conjugate to the standard affine action. By Lemma \ref{l:noC1} there is a copy of $\BS(1,2)$, call it $B\leq H$, that preserves a compact interval $I$ and, moreover, the action of $B$ on $I$ is not conjugate to an affine action (it is a Conradian action in the terminology of Section \ref{sec construction}).  In particular, no action of $H$ on the line is conjugate to an action by $C^1$-diffeomorphisms. 
\end{proof}

Let us next prove the lemma.

\begin{proof}[Proof of Lemma~\ref{l:noC1}]
For the proof it will be convenient to define $H=\langle a_i \mid a_i a_{i+1} a_i^{-1}=a_{i+1}^2\rangle$ where the indices are taken over $\Z/4\Z$. Let $\psi: H\to \HR_+(\R)$ be a non-trivial representation.

\begin{claim}\label{claim1}
The set of fixed points of each $\psi(a_i)$ accumulates on $\pm \infty$.
\end{claim}

\begin{proof}[Proof of Claim]
We argue by contradiction. We will assume that the fixed points of one generator do not accumulate on $+\infty$; the case of fixed points not accumulating on $-\infty$ is treated similarly. After re-indexation, we may also assume that this generator is $a_3$.  Since for orientation preserving actions on the line finite orbits are necessarily fixed points, the relation $a_2a_3a_2^{-1}=a_3^2$ implies that the subset
\[\mathsf{Fix}(\psi(a_3))=\{x\in \R\mid \psi(a_3)(x)=x\}\]
is $\psi(a_2)$-invariant. Therefore, $\psi(a_2)$ maps any connected component of the complement $\mathsf{supp}(\psi(a_3))=\R\setminus \mathsf{Fix}(\psi(a_3))$ to a (not necessarily different) connected component of $\mathsf{supp}(\psi(a_3))$.

Let $I$ be an unbounded component of the support of $\psi(a_3)$ accumulating on $+\infty$. Then, by the previous remark, $\psi(a_2)$ must either preserve or move $I$ off itself. Since $I$ is unbounded, we conclude that $\psi(a_2)$ preserves $I$. It follows from \cite{rivas} that in this case  the action of $\langle \psi(a_2),\psi(a_3)\rangle$ on $I$ is \emph{semi-conjugate} to the standard affine action of $\BS(1,2)$ on $\R$ (i.e. where $a_3$ acts as a translation and $a_2$ as dilation by 2), namely there exists a continuous, non-decreasing and surjective map $\alpha: I\to \R$ such that 
\[\alpha\circ  \psi(a_2) (x)= 2\alpha(x)\quad \text{and}\quad \alpha \circ \psi(a_3)(x)=\alpha(x)\pm 1.\]
The sign of translation depends on the action.
If $\alpha \circ \psi(a_3)(x)=\alpha(x)+ 1$, then $\psi(a_2)(x)>\psi(a_3)(x)>x$ for every $x$ large enough. If $\alpha \circ \psi(a_3)(x)=\alpha(x)- 1$, then $\psi(a_2)(x)>x>\psi(a_3)(x)$ for every $x$ large enough.
In both cases, the set of fixed points of $\psi(a_2)$ does not accumulate on $+\infty$. We can therefore repeat the argument with $a_2$ and $a_1$, then with $a_1$ and $a_0$ and then with $a_0$ and $a_3$ (this time $\psi(a_3)$ acting  as dilation), to get
\[\psi(a_3)(x)<\psi(a_2)(x)<\psi(a_1)(x)<\psi(a_0)(x)<\psi(a_3)(x)\]
for all $x$ large enough. Hence we obtain a contradiction.
\end{proof}

Observe that no $\psi(a_i)$ can be the identity, otherwise the image $\psi(H)$ would be trivial. Therefore every open support $\supp(\psi(a_i))=\{x\in \R\mid \psi(a_i)(x)\neq x\}$ is non-empty, and by Claim~\ref{claim1} all connected components are bounded intervals.

\begin{claim}\label{claim2}
	There is a pair of consecutive generators $a_i,a_{i+1}$ such that $\psi(a_i)$ moves off itself a connected component of the open support of $\psi(a_{i+1})$. 
\end{claim}

\begin{proof}[Proof of Claim]
Let $I_3$ be a connected component of $\supp(\psi(a_3))$. Arguing as in the proof of Claim \ref{claim1}, we see that if $\psi(a_2)$ does not preserve $I_3$, then it must move it off itself, so we are done. So suppose that $\psi(a_2)$ fixes $I_3$. Then the action of $\langle\psi(a_2),\psi(a_3)\rangle$ on $I_3$ is semi-conjugate to the standard affine action of $\BS(1,2)$ (where $\psi(a_3)$ is acting as the translation), and hence there is a connected component of $\supp(\psi(a_2))$  which is strictly contained in $I_3$. If $\psi(a_1)$ does not moves $I_2$ off itself, then we find  a connected component $I_1\subsetneq I_2$ of $\supp(\psi(a_1))$. Again, if $\psi(a_0)$ does not moves $I_1$ off itself, we find a connected component $I_0\subsetneq I_1\subsetneq I_2 \subsetneq I_3$ of $\supp(\psi(a_0))$. But then $I_0\subsetneq\supp(\psi(a_3))$, so $\psi(a_3)$ cannot preserve $I_0$, and therefore it must move $I_0$ off itself.
\end{proof}

Let $i$ be the index given by Claim \ref{claim2}, and let $J$ be a connected component of $\supp(\psi(a_{i+1}))$ which is moved off itself by $\psi(a_i)$. 
Let $I\subset\R$ be the smallest interval containing all the images $\psi(a_i)^n(J)$, $n\in \Z$. Claim~\ref{claim1} applied to $a_{i}$ implies that $I$ is bounded, and moreover its choice gives that it is preserved by both $\psi(a_i)$ and $\psi(a_{i+1})$.
Clearly the restriction of the action of $\langle \psi(a_i),\psi(a_{i+1})\rangle\cong \BS(1,2)$ to $I$ is faithful, and is not semi-conjugate to the standard affine action.
\end{proof}

We can finally conclude.

\begin{proof}[Proof of Theorem~\ref{t:noC1}]
	A direct consequence of Proposition~\ref{p:noC1} is that the analogous result holds for the circle, namely that every representation $\psi: H\to \Diff(\mathbb S^1)$ is trivial. This follows from the well-known fact that Higman's group is acyclic \cite{acyclic}, in particular $H^2(H,\Z)=0$ and every (continuous) action on the circle lifts to an action on the real line (the obstruction to such a lift is represented by the Euler class of the action, which is a $2$-cocycle \cite{ghys}).
	
	We can also use a more dynamical argument.
	For this, suppose such $\psi$ is given. 
	
	\begin{claim}
		For every $i\in \Z/4\Z$, the group $\langle \psi(a_i),\psi(a_{i+1})\rangle$ has a global fixed point in $\mathbb S^1$.
	\end{claim}
	\begin{proof}[Proof of Claim]
	Observe that the group $\langle a_i,a_{i+1}\rangle\cong\BS(1,2)$ is amenable. This implies that there exists a Borel probability measure $\mu_i$ on $\mathbb S^1$ which is invariant by the group $\langle \psi(a_i),\psi(a_{i+1})\rangle$, and moreover, the rotation number restricts to a group homomorphism $\rot:\langle \psi(a_i),\psi(a_{i+1})\rangle\to \R/\Z$ (see \cite{ghys}). We deduce that \[\rot\psi(a_{i+1})=\rot\psi(a_ia_{i+1}a_i^{-1})=\rot\psi(a_{i+1}^2)=2\rot\psi(a_{i+1}),\]
	so that $\rot\psi(a_{i+1})=0$, implying that $\psi(a_{i+1})$ has a fixed point. Cycling over the indices, we get that also $\psi(a_i)$ has a fixed point, and thus the probability measure $\mu_i$ must be supported on points that are fixed by both $\psi(a_i)$ and $\psi(a_{i+1})$, and thus by $\langle \psi(a_i),\psi(a_{i+1})\rangle$.
	\end{proof}
	
	From \cite{BMNR} we know that for every maximal interval $I\subset \mathbb S^1$ on which  $\langle \psi(a_i),\psi(a_{i+1})\rangle$ has no global fixed point, we either have \begin{enumerate}
		\item $\psi(a_{i+1})$ is the identity in restriction to $I$ or
		\item the restriction to $I$ of the action of $\langle \psi(a_i),\psi(a_{i+1})\rangle$ is conjugate to the standard affine action, and moreover the derivative of $\psi(a_i)$ at the internal fixed point equals $2$.
	\end{enumerate} 
	Continuity of derivatives implies that there are only finitely many intervals of the second type. Therefore $\supp(\psi(a_{i+1}))$ has finitely many connected components. Repeating the argument with the subgroup $\langle \psi(a_{i+1}),\psi(a_{i+2})\rangle$, we obtain that not only $\supp(\psi(a_{i+2}))$ has finitely many components, but moreover, they are strictly fewer than those of $\supp(\psi(a_{i+1}))$. Cycling over the indices, one gets a contradiction.
\end{proof}


\section*{Acknowledgments} 
C.R. would like to thank the organizers of Groups of Dynamical Origin II where he learned about the question of orderability of Higman's group. In particular, he is grateful to Yves de Cornulier for sharing this question and to Andr\'es Navas for his constant encouragement.  M.T. acknowledges the hospitality of USACH wile part of this project was carried  on.

\bibliographystyle{amsplain}


\begin{dajauthors}
\begin{authorinfo}[rivas]
	Crist\'obal Rivas\\
Dpto.~de Matem\'aticas y C.C., Universidad de Santiago de Chile\\
Alameda 3363, Estaci\'on Central, Santiago, Chile\\
cristobal.rivas\imageat{}usach\imagedot{}cl\\
{\footnotesize\url{http://www.mat.usach.cl/index.php/2012-12-19-12-50-19/academicos/183-cristobal-rivas}}
\end{authorinfo}
\begin{authorinfo}[tries]
	Michele Triestino\\
	Institut de Math\'ematiques de Bourgogne (IMB, UMR 5584)\\
9 av.~Alain Savary, 21000 Dijon, France\\
michele.triestino\imageat{}u-bourgogne\imagedot{}fr\\
\url{http://mtriestino.perso.cnrs.fr}
\end{authorinfo}
\end{dajauthors}

\end{document}